\documentclass[a4paper, 12pt]{article}
\usepackage[english]{babel}          
\usepackage{amsmath, amsfonts, amssymb, amsthm}

\newcounter{num}[section]

\newenvironment{theorem}
{\refstepcounter{num}%
\bigskip\noindent\nopagebreak[4]{\bf Theorem~\arabic{section}.\arabic{num}. }\it}
\newenvironment{corollary}
{\refstepcounter{num}%
\bigskip\noindent\nopagebreak[4]{\bf Corollary~\arabic{section}.\arabic{num}. }\it}
\newenvironment{lemma}
{\refstepcounter{num}%
\bigskip\noindent\nopagebreak[4]{\bf Lemma~\arabic{section}.\arabic{num}. }\it}
\newenvironment{remark}
{\refstepcounter{num}%
\bigskip\noindent\nopagebreak[4]{\bf Remark~\arabic{section}.\arabic{num}. }}

\newenvironment{example}
{\refstepcounter{num}%
\bigskip\noindent\nopagebreak[4]{\bf Example~\arabic{section}.\arabic{num}. }}

\newcommand{\Zbb}{{\mathbb{Z}}}

\newcommand{\LL}{{\mathcal{L}}}

\newcommand{\Ss}{{\mathbf{S}}}

\newcommand{\V}{{\mathrm{V}}}
\newcommand{\qq}{{\mathrm{q}_\omega}}
\newcommand{\uu}{{\mathrm{u}_\omega}}

\newcommand{\lb}{{\langle}}
\newcommand{\rb}{{\rangle}}

\renewcommand{\c}{{\mathbf{c}}}

\newcommand{\Acal}{{\mathcal{A}}}

\newcommand{\1}{{^{-1}}}
\newcommand{\eps}{{\varepsilon}}

\renewcommand{\b}{{\mathbf{b}}}

\renewcommand{\a}{{\mathbf{a}}}

\newcommand{\s}{{\sigma}}

\newcommand{\pr}{{\prime}}

\newcommand{\Aut}{{\mathrm{Aut}}}
\newcommand{\Inn}{{\mathrm{Inn}}}

\begin{document}

\author{Artem N. Shevlyakov\footnote{The author was supported by Russian Fund of Fundamental Research (project 18-31-00330, the results of Sections 3)  and Russian Science Foundation (project 18-71-10028, the results of Section 4) }}

\title{On group automorphisms in universal algebraic geometry}

\maketitle

\abstract{In this paper we study group equations with occurrences of automorphisms. We describe equational domains in this class of equations. Moreover, we solve a number of open problem posed in universal algebraic geometry}

\section{Introduction}

In the classic approach to algebraic geometry over groups we are dealing with equations over a group $G$ as expressions $w(X)=1$, where $w(X)$ is an element of $G\ast F(X)$ ($F(X)$ is the free group generated by a set of variables $X$). This class of equations was studied in many papers (see~\cite{BMR1,DMR_monograph} and the survey~\cite{romankov_survey} for more details). 

In the current paper we consider a different class of equations over a group $G$:  now $w(X)$ may contain the occurrences of symbols $\{\phi\mid \phi\in \Aut(G)\}$. Any equation of this type is called below an {\it equation with automorphisms}. The study of such equations is justified by many important problems in group theory. For example, the twisted conjugacy problem for a group $G$  is equivalent to the solution of the following equation $\phi(x)u=vx$ for given $u,v\in G$, $\phi\in\Aut(G)$ (also, see this problem in~\cite{romankov_nilpotent} for equations with endomorphisms). 

There is a connection between the ``standard'' group equations and equations with automorphisms. Indeed, for an equation $c_0x_1c_1x_2c_2\ldots c_{k-1}x_kc_k=1$ ($c_i\in G$) we have
\begin{multline*}
c_0x_1c_1x_2c_2\ldots c_{k-1}x_kc_k=(c_0x_1c_0^{-1})(c_0c_1x_2c_1^{-1}c_0^{-1})\ldots\\
 (c_0c_1c_2\ldots c_{k-1}x_kc_{k-1}^{-1}\ldots c_2^{-1}c_1^{-1}c_0^{-1})c_0c_1c_2\ldots c_k=
x_1^{c_0^{-1}}x_2^{(c_0c_1)^{-1}}\ldots x_k^{(c_0c_1c_2\ldots c_{k-1})^{-1}}\prod c_i\\
=\phi_1(x_1)\phi_2(x_2)\ldots \phi_k(x_k)\prod c_i,
\end{multline*}
where each $\phi_i$ is an inner automorphism of a group $G$. Thus, the equation above  is equivalent to the following equation with automorphisms
\[
\phi_1(x_1)\phi_2(x_2)\ldots \phi_k(x_k)c=1\; (c\in G).
\]

This correspondence allows us to study equations with automorphisms by methods developed for group equations. For instance, in~\cite{DMR_monograph} equational domains for group equations were described. In Section~\ref{sec:ed} we solve the similar problem for equations with automorphisms. The results of Section~\ref{sec:ed} provide various examples of equational domains, so it allows to solve a number of open problems posed in universal algebraic geometry (Section~\ref{sec:problem}). Namely, we solve Problem~4.4.7 from~\cite{DMR_monograph}. Notice that our solution implies negative answers for Problems~5.3.1-4 in~\cite{DMR_monograph} (the reduction of Problem~4.4.7 to Problems~5.3.1-4 was shown in~\cite{DMR_monograph}).

\section{Definitions}

All definitions below are derived from~\cite{DMR_monograph}, where all notions of algebraic geometry were formulated for algebraic structures of arbitrary languages.

Denote by $\LL=\{\cdot,^{-1},1\}$ the standard language of group theory. Let us fix a group $G$ and consider 
the extended language $\LL(A)=\{\cdot,^{-1},1\}\cup\{\phi^{(1)}\mid \phi^{(1)}\in A\}$, 
where the unary functional symbols $\phi^{(1)}$  correspond to a group of  automorphisms $A\subseteq \Aut(G)$. Any group $G$ of the language $\LL(A)$ is called an $\LL(A)$-group (implicitly we fix an interpretation of the symbols $\phi$ to the elements of the group $A\subseteq \Aut(G)$).

Using the properties of automorphisms, any $\LL(A)$-term in variables  $X=\{x_1,x_2,\ldots,x_n\}$ is equivalent to a product
\begin{equation}
\phi_1(x_{i_1}^{\eps_1})\phi_2(x_{i_2}^{\eps_2})\ldots\phi_k(x_{i_k}^{\eps_k}),
\label{eq:L(A)-term}
\end{equation}
where $\phi_j\in A$, $x_{i_j}\in X$, $\eps_j\in \{-1,1\}$.

An {\it $\LL(A)$-equation} is an expression $t(X)=1$, where $t(X)$ is an $\LL(A)$-term. {\it An $\LL(A)$-system} is an arbitrary set of $\LL(A)$-equations. The set of all solutions of an $\LL(A)$-system $\Ss$ in $G$ is denoted by $\V_G(\Ss)$. A set $Y\subseteq G^n$ is called  {\it $\LL(A)$-algebraic} if there exists an $\LL(A)$-system $\Ss$ in variables $X=\{x_1,x_2,\ldots,x_n\}$ with $Y=\V_G(\Ss)$. 

An $\LL(A)$-group $G$ is an {\it $\LL(A)$-equational domain} if for any $n$ and arbitrary $\LL(A)$-algebraic sets $Y_1,Y_2\subseteq G^n$ the union $Y=Y_1\cup Y_2$ is also $\LL(A)$-algebraic. 

\begin{theorem}\textup{(\cite{DMR_monograph})}
\label{th:criterion_ED}
An $\LL(A)$-group $G$ is an $\LL(A)$-equational domain iff there exists an $\LL(A)$-system $\Ss$ in variables $x,y$ such that 
\[
\V_G(\Ss)=\{(x,y)\mid x=1\mbox{ or }y=1\}.
\]
\end{theorem}

\begin{remark}
\label{rem:1}
Actually, Theorem~\ref{th:criterion_ED} was proved for group languages with constants, but its proof is valid for arbitrary group languages. 
\end{remark}

\bigskip

Let us recall the results of~\cite{DMR_monograph} related to group equations with no automorphisms. 

Let $H$ be a fixed subgroup of a group $G$. We pick elements of $H$ as constants in the language $\LL(H)=\LL\cup\{h\mid h\in H\}$. Any $\LL(H)$-term in variables $X$ is actually an element of the free product $H\ast F(X)$, where $F(X)$ is the free group generated by the set $X$. An $\LL(H)$-equation is an expression $t(X)=1$, where $t(X)$ is an $\LL(H)$-term. Naturally, one can define the notions of algebraic sets and equational domains in the language $\LL(H)$. As we mentioned in Remark~\ref{rem:1}, Theorem~\ref{th:criterion_ED} holds for $\LL(H)$-equational domains.

It was found in~\cite{DMR_monograph} the complete description of $\LL(H)$-equational domains. 

\begin{theorem}\textup{(\cite{DMR_monograph})}
An $\LL(H)$-group $G$ is an $\LL(H)$-equational domain iff there are not $a,b\in G$, $a,b\neq 1$ such that  
\begin{equation}
[a,h^{-1}bh]=1\mbox{ for any } h\in H
\label{eq:zero_div_condition_classic}
\end{equation}
(here $[x,y]=x^{-1}y^{-1}xy$).
\label{th:about_zero_divisors_classic}
\end{theorem}

\medskip

According to Theorem~\ref{th:about_zero_divisors_classic}, one can obtain (see~\cite{DMR_monograph}) few examples of $\LL(H)$-equational domains: 
\begin{enumerate}
\item the free group $F_2$ of rank $2$ (for $H=F_2$),
\item the alternating group $A_5$ (for $H=A_5$).
\end{enumerate}
Both examples will be used below in the paper.

\section{Equational domains}
\label{sec:ed}

Let us study equations with automorphisms, and the following theorem describes equational domains in the class of $\LL(A)$-groups. Its proof is similar to Theorem~\ref{th:about_zero_divisors_classic} from~\cite{DMR_monograph}.

\begin{theorem}
An $\LL(A)$-group $G$ is an $\LL(A)$-equational domain iff there are not $a,b\in G$, $a,b\neq 1$ such that  
\begin{equation}
[a,\phi(b)]=1\mbox{ for any } \phi\in A. 
\label{eq:zero_div_condition}
\end{equation}
\label{th:about_zero_divisors}
\end{theorem}
\begin{proof}
Let us prove the ``if'' statement. We have that any solution $(x,y)$ of the $\LL(A)$-system $\Ss=\{[x,\phi(y)]=1\mid \phi\in A\}$ satisfies $x=1$ or $y=1$. Thus, $\V_G(\Ss)=\{(x,y)\mid x=1\mbox{ or }y=1\}$, and Theorem~\ref{th:criterion_ED} concludes the proof.

Now, we prove  the ``only if'' part of the theorem. By Theorem~\ref{th:criterion_ED}, there exists an $\LL(A)$-system $\Ss$ with the solution set $\{(x,y)\mid x=1\mbox{ or }y=1\}$. Let $w(x,y)=1$ be an arbitrary $\LL(A)$-equation of $\Ss$. Using the following commutator identities,
\[
[s,t]^{-1}=[t,s],\; [sp,t]=[s,t]^p[p,t],\; [s^{-1},t]=[t,s]^{s^{-1}}
\]
one can equivalently rewrite $w(x,y)$ as a product
\[
w(x,y)=u(x)v(y)\prod_i [\phi_i(x),\psi_i(y)]^{w_i(x,y)}
\] 
where $\phi_i,\psi_i\in A$, and $w_i(x,y)$, $u(x)$, $v(y)$ are $\LL(A)$-terms.

Since $(1,y),(x,1)\in\V_G(\Ss)$ for any $x,y\in G$, then $u(x)=1$ and $v(y)=1$ for all $x,y\in G$. Hence, one can assume that any equation $w(x,y)=1\in\Ss$ is of the form
\[
\prod_i [\phi_i(x),\psi_i(y)]^{w_i(x,y)}=1.
\] 

Assume there exist $a,b\in G$, $a,b\neq 1$ with~(\ref{eq:zero_div_condition}). We have
\begin{multline*}
[a,\phi_i\1(\psi_i(b))]=1\Leftrightarrow a\phi_i\1(\psi_i(b))=\phi_i\1(\psi_i(b))a\Leftrightarrow \\
\phi_i(a\phi_i\1(\psi_i(b)))=\phi_i(\phi_i\1(\psi_i(b))a)\Leftrightarrow \\
\phi_i(a)\psi_i(b)=\psi_i(b)\phi_i(a)\Leftrightarrow [\phi_i(a),\psi_i(b)]=1
\end{multline*}
and $(a,b)\in\V_G(w(x,y)=1)$. Thus, the point $(a,b)$ satisfies any equation of $\Ss$, and we obtain a contradiction $\V_G(\Ss)\neq \{(x,y)\mid x=1\mbox{ or }y=1\}$.
\end{proof}

Let us compare Theorem~\ref{th:about_zero_divisors} and Theorem~\ref{th:about_zero_divisors_classic}. Obviously, Theorem~\ref{th:about_zero_divisors_classic} follows from Theorem~\ref{th:about_zero_divisors} for $A=\Inn_H(G)$, where $\Inn_H(G)=\{\phi_h(g)=h^{-1}gh\mid h\in H\}$ is a subgroup of the group $\Inn(G)$ of inner automorphisms. 

Moreover, if a group $G$ is an $\LL(H)$-equational domain, then $G$ is an $\LL(A)$-equational domain for any $A\supseteq\Inn_H(G)$. Therefore, the alternating group $A_5$ is an $\LL(A)$-equational domain for $A=\Aut(A_5)$. The free group $F_2$ of rank $2$ is also an $\LL(A)$-equational domain for $A=\Aut(F_2)$. However, the following statement provides $F_2$ to be an equational domain with a cyclic group of automorphisms.

\begin{example}
Let $F_2$ be the free group of rank $2$, and $a,b$ be free generators. Let $\phi$ denote the automorphism $\phi(a)=b$, $\phi(b)=a$. Then Theorem~\ref{th:about_zero_divisors} states that $F_2$ is an $\LL(A)$-equational domain for $A=\lb \phi\rb$.
\end{example}

\medskip

The following two statements also follow from Theorem~\ref{th:about_zero_divisors}.

\begin{corollary}
If a group $G$ has a nontrivial center $Z(G)$, then $G$ is not an $\LL(A)$-equational domain for any $A\subseteq \Aut(G)$.
\end{corollary}
\begin{proof}
Let $a\in Z(G)\setminus\{1\}$ be a central element. Hence, $a$ commute with any $\phi(a)$, and the pair $(a,a)$ satisfies~(\ref{eq:zero_div_condition}) for all $\phi$.
\end{proof}

\begin{corollary}
Let $G$ be an $\LL(A_0)$-equational domain for some $A_0\subseteq \Aut(G)$, and $H=\Pi G$ be a direct power of $G$ indexed by a set $I$. In other words, any element of $H$ is an ordered tuple $(g_i\mid i\in I)$. Let $\mathcal{P}$ be a set of permutations of $I$ such that $\mathcal{P}$ is transitive on $I$ (i.e. for any pair $i,j\in I$ there exists $\pi\in\mathcal{P}$ with $\pi(i)=j$). Let us define automorphisms of $H$ as follows:
\begin{equation}
\label{eq:f_phi}
f_\phi((g_i \mid i\in I))=(\phi(g_i) \mid i\in I),
\end{equation}
\begin{equation}
\label{eq:sigma_phi}
\s_\pi((g_i \mid i\in I))=(g_{\pi(i)} \mid i\in I),
\end{equation}
where $\phi\in A_0$, $\pi\in \mathcal{P}$.
Let $A\subseteq \Aut(H)$ denote the group generated by $\{f_\phi,\s_\pi\mid \phi\in A_0,\pi\in \mathcal{P}\}$. Then the $\LL(A)$-group $H$ is an $\LL(A)$-equational domain.
\label{cor:ed_for_direct_powers}
\end{corollary}
\begin{proof}
Let us take $\a=(a_i\mid i\in I),\b=(b_i\mid i\in I)\in H$, $\a,\b\neq 1$. 
Since $\mathcal{P}$ transitively acts on $I$, there exists $\psi\in A$ and an index $i\in I$ such that $a_i\neq 1$, $c_i\neq 1$ where $\psi(\b)=\c=(c_i\mid i\in I)$.

Since $G$ is an $\LL(A_0)$-equational domain, there exists $\phi\in A_0$ with
\[
[a_i,\phi(c_i)]\neq 1.
\] 
Therefore, 
\[
[\a,f_\phi(\psi(\b))]\neq 1,
\] 
and Theorem~\ref{th:about_zero_divisors} completes the proof.

\end{proof}

\section{One problem from universal algebraic geometry}
\label{sec:problem}

The book~\cite{DMR_monograph} contains an open problem (Problem~4.4.7), which can be equivalently formulated as follows: is there an algebraic structure $\Acal$ of an appropriate language $L$ such that
\begin{enumerate}
\item $\Acal$ is an $L$-equational domain;
\item $\Acal$ is $\qq$-compact;
\item $\Acal$ is not $\uu$-compact.
\end{enumerate}

We solve this problem in the class of $\LL(A)$-groups. Let us give all necessary definitions.

Let $A\subseteq\Aut(H)$ be a subgroup of automorphisms of a group $H$. An $\LL(A)$-group $H$ is {\it $\qq$-compact} if for any $\LL(A)$-system $\Ss$ and an $\LL(A)$-equation $w(X)=1$ such that
\begin{equation}
\V_{H}(\Ss)\subseteq \V_{H}(w(X)=1)
\label{eq:q_compact1}
\end{equation}
there exists a finite subsystem $\Ss^\prime\subseteq\Ss$ with
\begin{equation}
\V_{H}(\Ss^\pr)\subseteq \V_{H}(w(X)=1).
\label{eq:q_compact2}
\end{equation}

An $\LL(A)$-group $H$ is {\it $\uu$-compact} if for any $\LL(A)$-system $\Ss$ and $\LL(A)$-equations $w_i(X)=1$ ($1\leq i\leq m$) such that
\begin{equation}
\V_{H}(\Ss)\subseteq \bigcup_{i=1}^m \V_{H}(w_i(X)=1)
\label{eq:u_compact1}
\end{equation}
there exists a finite subsystem $\Ss^\prime\subseteq\Ss$ with
\begin{equation}
\V_{H}(\Ss^\pr)\subseteq \bigcup_{i=1}^m \V_{H}(w_i(X)=1)
\label{eq:u_compact2}
\end{equation}

Let us define a group solving the problem above. Let $G$ be a finite group such that $G$ is an $\LL(A_0)$-equational domain for $A_0=\Aut(G)$ (for example, one may take $G=A_5$). Following Corollary~\ref{cor:ed_for_direct_powers}, we define the $\LL(A)$-group $H=\Pi G$ for $I=\mathbb{Z}$, $\mathcal{P}=\{\pi\}$ (where $\pi$ is a permutation $\pi(n)=n+1$ over $\mathbb{Z}$), and $A$ is generated by the automorphisms $f_\phi,\s_\pi$~(\ref{eq:f_phi},\ref{eq:sigma_phi}). 

We denote the subgroup generated by $\{f_\phi\mid \phi\in \Aut(G)\}\subseteq A$  by $A_G$. The automorphism $\s_\pi$ is denoted by $\s$ below. By the definition, $\s$ acts on an element $(g_i\mid i\in\Zbb)$ by 
\[
\s(g_i)=g_{i+1}.
\]

%By $H=\Pi G$ we denote the direct power of $G$ indexed by integers, i.e. $H=\{(\ldots,g_{-2},g_{-1},g_0,g_1,g_2,\ldots)\mid g_i\in G\}$. The following shift 

%is an automorphism of $H$.
%
%Obviously, any $\phi\in \Aut(G)$ induces the automorphism $f_\phi$ of $H$ by 
%\[
%f_\phi((\ldots,g_{-2},g_{-1},g_0,g_1,g_2,\ldots))=(\ldots,\phi(g_{-2}),\phi(g_{-1}),\phi(g_0),\phi(g_1),\phi(g_2),\ldots).
%\]
%Let $A$ denote the group of automorphisms generated by $\{\s\}\cup\{f_\phi\mid \phi\in \Aut(G)\}$ and below we consider $H$ as an $\LL(A)$-group. 

Thus, we should prove that $H$ is 
\begin{enumerate}
\item an $\LL(A)$-equational domain (it immediately follows from Corollary~\ref{cor:ed_for_direct_powers});
\item $\qq$-compact (Lemma~\ref{l:main});
\item not $\uu$-compact (Lemma~\ref{l:u_compact}).
\end{enumerate}

Below we will use the following denotation
\[
\s_k(x)=
\begin{cases}
\underbrace{\s(\s(\ldots\s}_{\mbox{$k$ times}}(x)\ldots))\mbox{ for }k>0,\\
\underbrace{\s^{-1}(\s^{-1}(\ldots\s^{-1}}_{\mbox{$k$ times}}(x)\ldots))\mbox{ for }k<0,\\
x\mbox{ for }k=0
\end{cases}
\]

The automorphism $\s$ commute with any $f_\phi$, i.e. $\s(f_\phi(h))=f_\phi(\s(h))$ for all $h\in H$. Hence any equation over the $\LL(A)$-group $H$ can be written in the following form
\begin{equation}
\s_{k_1}(f_1(x_{j_1}^{\eps_1}))\s_{k_2}(f_2(x_{j_2}^{\eps_2}))\ldots\s_{k_l}(f_l(x_{j_l}^{\eps_l}))=1,
\label{eq:equation_over_H}
\end{equation}
where $f_i\in A_G$, $\eps_i\in\{-1,1\}$, $k_j\in\Zbb$.

\begin{lemma}
\label{l:u_compact}
The $\LL(A)$-group $H$ is not $\uu$-compact.
\end{lemma}
\begin{proof}
Since $H$ is an $\LL(A)$-equational domain, there are no $\a,\b\in H$ such that $\a,\b\neq 1$ and $[\a,\phi(\b)]=1$ for any $\phi\in A$. Hence, any solution of the $\LL(A)$-system $\Ss=\{[x,\phi(y)]=1 \mid \phi\in A\}$ satisfies either $x=1$ or $y=1$. Thus, the following inclusion 
\[
\V_H(\Ss)\subseteq\V_H(x=1)\cup \V_H(y=1).
\]
holds.

Let $\Ss^\pr$ be a finite subsystem of $\Ss$ and $n=\max\{|k|\mid \s_k\mbox{ occurs in }\Ss^\pr\}$.
Define $\a=(a_i\mid i\in\Zbb)$, $\b=(b_i\mid i\in \Zbb)$ such that
\[
a_i=\begin{cases}
g,\mbox{ if }i=0\\
1,\mbox{ otherwise}
\end{cases}
\; 
b_i=\begin{cases}
g,\mbox{ if }i=n+1\\
1,\mbox{ otherwise}
\end{cases}
\]
where $g\in G\setminus\{1\}$.

Let $A^\pr=\{\phi\mid [x,\phi(y)]=1\in\Ss^\pr\}$ (i.e. $A^\pr$ is the set of all $\phi$ such that the equation $[x,\phi(y)]=1$ belongs to $S^\pr$) be a finite set of automorphisms. By the choice of $\b$, the element $\phi(\b)$ has $1$ at the 0-th coordinate for each $\phi\in A^\pr$. Therefore, $\phi(\b)$ commutes with $\a$ and we obtain $(\a,\b)\in \V_H(\Ss^\pr)$. Since $\a\neq 1$, $\b\neq 1$, then the inclusion  
\[
\V_H(\Ss^\pr)\subseteq\V_H(x=1)\cup \V_H(y=1)
\]
fails. Thus, $H$ is not $\uu$-compact.
\end{proof}

There is a correspondence between $\LL(A)$-systems over $H$ and $\LL(A_0)$-systems over $G$. Let $\Ss$ be an $\LL(A)$-system in variables $X=\{x_1,x_2,\ldots,x_n\}$. The system $\Ss$ defines an $\LL(A_0)$-system $\gamma(\Ss)$ over $G$ in infinite number of variables $Y=\{y_{ij}\mid i\in\Zbb, 1\leq j\leq n\}$ (below $\eps_i\in\{-1,1\}$):
\begin{multline}
\label{eq:gamma_S}
\s_{k_1}(f_1(x_{j_1}^{\eps_1}))\s_{k_2}(f_2(x_{j_2}^{\eps_2}))\ldots\s_{k_l}(f_l(x_{j_l}^{\eps_l}))=1
\Leftrightarrow\\
f_1(y_{k_1+k\ j_1}^{\eps_1})f_2(y_{k_2+k\ j_2}^{\eps_2})\ldots f_l(y_{k_l+k\ j_l}^{\eps_l})=1\in \gamma(\Ss)\mbox{ for all $k\in \Zbb$}.
\end{multline}

In other words, $\gamma(\Ss)$ is the coordinate-wise version of $\Ss$ over the direct power $H=\Pi G$.

\begin{example}
If $\Ss=\{\s(x_1)x_2=1\}$ then 
\begin{multline*}
\gamma(\Ss)=\{\ldots,y_{-11}y_{-22}=1,y_{01}y_{-12}=1,y_{11}y_{02}=1,y_{21}y_{12}=1,y_{31}y_{22}=1,\ldots\}=\\
\{y_{k1}y_{(k-1)2}=1\mid k\in\Zbb\}.
\end{multline*}
\label{ex:ex}
\end{example}

By the definition, any $\LL(A_0)$-equation $u(Y)=1\in\gamma(\Ss)$ may come from several  $\LL(A)$-equations $W=\{w_i(X)=1\}$ of the system $\Ss$. Let us take an arbitrary equation $w_i(X)=1$ from $W$ and denote this correspondence by $\gamma^{-1}(u(Y)=1)=\{w_i(X)=1\}$.

\begin{remark}
Below we will omit brackets in map compositions, i.e. we will write $\alpha\beta(x)$ instead of $\alpha(\beta(x))$. 
\end{remark}

%By the definition of the maps $\gamma,\gamma^{-1}$, we have $\gamma^{-1}\gamma(\Ss)=\Ss$ for any $\LL(A)$-system $\Ss$.

\begin{lemma}
For any $\LL(A_0)$-equation 
\begin{equation}
f_1(y_{i_1\ j_1}^{\eps_1})f_2(y_{i_2\ j_2}^{\eps_2})\ldots f_l(y_{i_l\ j_l}^{\eps_l})=1
\label{eq:222}
\end{equation}
and any number $k\in\Zbb$ the equation 
\begin{equation}
f_1(y_{i_1+k\ j_1}^{\eps_1})f_2(y_{i_2+k\ j_2}^{\eps_2})\ldots f_l(y_{i_l+k\ j_l}^{\eps_l})=1
\label{eq:333}
\end{equation}
also belongs to $\gamma(\Ss)$.

Further, if $P=(p_{ij}\mid i\in\Zbb, 1\leq j\leq n)\in\V_G(\gamma(\Ss))$ then any shift $\s_k(P)=(s_{ij}\mid i\in\Zbb, 1\leq j\leq n)$, $s_{ij}=p_{i+k\ j}$ ($k\in\Zbb$) is also a solution of $\gamma(\Ss)$.
\label{l:homogeneity_of_system}
\end{lemma}
\begin{proof}
Observe that the system $\Ss$ from Example~\ref{ex:ex} clearly satisfies the statements of this lemma.

The first statement directly follows from the definition of the system $\gamma(\Ss)$. Let us prove the second one. 

Assume there exists an $\LL(A_0)$-equation~(\ref{eq:222}) with $f_1(s_{i_1\ j_1}^{\eps_1})f_2(s_{i_2\ j_2}^{\eps_2})\ldots f_l(s_{i_l\ j_l}^{\eps_l})\neq 1$ or, equivalently, 
\begin{equation}
\label{eq:1111111}
f_1(p_{i_1+k\ j_1}^{\eps_1})f_2(p_{i_2+k\ j_2}^{\eps_2})\ldots f_l(p_{i_l+k\ j_l}^{\eps_l})\neq 1
\end{equation}

However, $\gamma(\Ss)$ contains the equation $u(Y)=1$~(\ref{eq:333}), and, by (\ref{eq:1111111}), we have $u(P)\neq 1\Rightarrow P\notin\V_G(\gamma(\Ss))$.
\end{proof}

Let $\Ss_0,\Ss_1$ be $\LL(A_0)$-systems in variables $Y=\{y_{ij}\mid i\in\Zbb, 1\leq j\leq n\}$. We say that $\Ss_0,\Ss_1$ are {\it $Z$-equivalent} for a given $Z\subseteq Y$ if the projections of $\V_G(\Ss_0)$ and $\V_G(\Ss_1)$ onto the coordinates $Z$ are the same (in other words, 
for each $P=(p_{ij}\mid i\in\Zbb, 1\leq j\leq n)\in\V_G(\Ss_k)$ there exists $Q=(q_{ij}\mid i\in\Zbb, 1\leq j\leq n)\in\V_G(\Ss_{1-k})$ with $p_{ij}=q_{ij}$ for each $y_{ij}\in Z$, $k\in\{0,1\}$.

\begin{lemma}
Let $\Ss_0$ be an $\LL(A_0)$-system in variables $Y=\{y_{ij}\mid i\in\Zbb, 1\leq j\leq n\}$ over a finite group $G$. Then for any finite $Z\subseteq Y$ there exists a finite $Z$-equivalent subsystem $\Ss_1\subseteq\Ss_0$.
\end{lemma}  
\begin{proof}
The statement immediately follows from the finiteness of the group $G$.
\end{proof}

%Below we take 
%\begin{equation}
%Z=\{y_{01},y_{02},\ldots,y_{0n}\}
%\label{eq:Z}
%\end{equation}
%and denote consider $\gamma_Z(\Ss)$ for such $Z$

Let us denote a subsystem of an $\LL(A_0)$-system $\gamma(\Ss)$ by $\gamma_Z(\Ss)$, if $\gamma_Z(\Ss)$ is $Z$-equivalent to $\gamma(\Ss)$.

Let $Z$ be a set of variables occurring in an $\LL(A_0)$-system $\gamma(\Ss)$. The system $\gamma(\Ss)$ may contain subsystems which are $Z$-equivalent to $\gamma(\Ss)$ (as it proved above, for finite $Z$ such subsystems always exist). Let us denote the class of such systems by $Z(\gamma(\Ss))$. We pick an arbitrary system from $Z(\gamma(\Ss))$ and denote it by $\gamma_Z(\Ss)$.

Suppose an $\LL(A_0)$-system $\gamma_Z(\Ss)$ was constructed by an $\LL(A)$-system $\Ss$ and a finite set $Z$. By the definition, $\gamma^{-1}\gamma_Z(\Ss)\subseteq \Ss$ is the set of equations from $\Ss$ which were essentially used in the construction of $\gamma_Z(\Ss)$. One can apply the operator $\gamma$ to $\gamma^{-1}\gamma_Z(\Ss)$ and obtain a new $\LL(A_0)$-system $\gamma\gamma^{-1}\gamma_Z(\Ss)$.

Let us summarize all simple properties of the systems $\Ss,\gamma(\Ss),\gamma_Z(\Ss),\gamma^{-1}\gamma_Z(\Ss),\gamma\gamma^{-1}\gamma_Z(\Ss)$:
\begin{enumerate}
\item $\Ss,\gamma^{-1}\gamma_Z(\Ss)$ are $\LL(A)$-systems and their solutions belong to $H^n$;
\item $\gamma(\Ss),\gamma_Z(\Ss),\gamma\gamma^{-1}\gamma_Z(\Ss)$ are $\LL(A_0)$-systems and coordinates of their solutions belong to $G$;
\item the systems $\gamma_Z(\Ss),\gamma^{-1}\gamma_Z(\Ss)$ are finite for finite $Z$;
\item we have the inclusions $\gamma^{-1}\gamma_Z(\Ss)\subseteq\Ss$, $\gamma_Z(\Ss)\subseteq\gamma\gamma^{-1}\gamma_Z(\Ss)\subseteq\gamma(\Ss)$.
\item the $\LL(A_0)$-systems $\gamma\gamma^{-1}\gamma_Z(\Ss)$, $\gamma_Z(\Ss)$, $\gamma(\Ss)$ are $Z$-equivalent.
\end{enumerate}

\begin{lemma}
Let $C$ be a finite set of pairs $(i,j)$, $i\in\Zbb$, $1\leq j\leq n$ Then $\gamma\gamma^{-1}\gamma_Z(\Ss)$ is $Z_k$-equivalent to $\gamma(\Ss)$ for any set $Z_k=\{y_{i+k\ j}\mid (i,j)\in C\}$, $k\in\Zbb$.
\label{l:Z_k-equivalence}
\end{lemma}
\begin{proof}
Let us take a point $P=(p_{ij}\mid i\in\Zbb,1\leq j\leq n)\in\V_G(\gamma\gamma^{-1}\gamma_Z(\Ss))$ and consider the shift $R=\s_k(P)=(r_{ij}\mid i\in\Zbb,1\leq j\leq n)$, $r_{ij}=p_{i+k\ j}$. According to Lemma~\ref{l:homogeneity_of_system}, $R$ is a solution of $\V_G(\gamma\gamma^{-1}\gamma_Z(\Ss))$. By the $Z$-equivalence, there exists a point $R^\pr=(r_{ij}^\pr\mid i\in\Zbb,1\leq j\leq n)\in\V_G(\gamma(\Ss))$ with $r_{ij}^\pr=r_{ij}=p_{i+k\ j}$ for any $(i,j)\in C$. By Lemma~\ref{l:homogeneity_of_system}, the point $R^{\pr\pr}=\s_{-k}(R^\pr)$, $R^{\pr\pr}=(r^{\pr\pr}_{ij}\mid i\in\Zbb,1\leq j\leq n)$ is a solution of   $\gamma(\Ss)$. By the definition of $R^{\pr\pr}$, for each $(i,j)\in C$ we have $r^{\pr\pr}_{i+k\ j}=r^\pr_{ij}=r_{ij}=p_{i+k\ j}$, and, therefore, $\gamma\gamma^{-1}\gamma_Z(\Ss)$ is $Z_k$-equivalent to $\gamma(\Ss)$.
\end{proof}

\begin{lemma}
The $\LL(A)$-group $H$ is $\qq$-compact.
\label{l:main}
\end{lemma}
\begin{proof}
Suppose an $\LL(A)$-system $\Ss$ and an $\LL(A)$-equation $w(X)=1$~(\ref{eq:equation_over_H}) satisfy~(\ref{eq:q_compact1}). 

The $\LL(A)$-term $w(X)$ defines the set of pairs 
\[
C=\{(k_1,j_1),(k_2,j_2),\ldots,(k_l,j_l)\}.
\]
%\s_{k_1}(f_1(x_{i_1}^{\eps_1}))\s_{k_2}(f_2(x_{i_2}^{\eps_2}))\ldots\s_{k_l}(f_l(x_{i_l}^{\eps_l}))=1,

Let us put $\Ss^\pr=\gamma^{-1}\gamma_Z(\Ss)$ for $Z=\{y_{ij}\mid (i,j)\in C\}$ and prove~(\ref{eq:q_compact2}). Assume there exists a point $(\mathbf{h}_1,\mathbf{h}_2,\ldots,\mathbf{h}_n)\in\V_H(\Ss^\pr)\setminus\V_H(w(X)=1)$, $\mathbf{h}_j\in H$. In other words, there exists $P=(p_{ij}\mid i\in\Zbb,1\leq j\leq n)\in\V_G(\gamma\gamma^{-1}\gamma_Z(\Ss))\setminus\V_G(\gamma(w(X))=1)$. 

We have
\[
\gamma(w(X)=1)=
\{f_1(y_{k_1+k\ j_1}^{\eps_1})f_2(y_{k_2+k\ j_2}^{\eps_2})\ldots f_l(y_{k_l+k\ j_l}^{\eps_l})=1\mid k\in \Zbb\}
\]
and there exists $k\in\Zbb$ such that
\[
f_1(p_{k_1+k\ j_1}^{\eps_1})f_2(p_{k_2+k\ j_2}^{\eps_2})\ldots f_l(p_{k_l+k\ j_l}^{\eps_l})\neq 1
\]

By Lemma~\ref{l:Z_k-equivalence}, there exists a point $Q=(q_{ij}\mid i\in\Zbb,1\leq j\leq n)\in\V_G(\gamma(\Ss))$ with $q_{i+k\ j}=p_{i+k\ j}$ for any $(i,j)\in C$. Therefore, $w(Q)\neq 1$.

Thus, $Q\in \V_G(\gamma(\Ss))\setminus\V_G(\gamma(w(X)=1))$. The point $Q$ defines $R=(\mathbf{r}_1,\mathbf{r}_2,\ldots,\mathbf{r}_n)\in H^n$, $\mathbf{r}_j=(q_{ij}\mid i\in \Zbb)$ such that $R\in \V_G(\Ss)\setminus\V_G(w(X)=1)$, and we obtain a contradiction with~(\ref{eq:q_compact1}).
\end{proof}

\section{Conclusions}

The construction of the group $H$ from Corollary~\ref{cor:ed_for_direct_powers} is close to the notion of wreath product. In particular, the group $H$ from Section~\ref{sec:problem} is structurally similar to the wreath product $G\wr \Zbb$. 

This correspondence allows us to remind an important problem of universal algebraic geometry posed by B.~Plotkin~\cite{plotkin_ispire}.

\medskip

\noindent{\bf Problem }(B.~Plotkin~\cite{plotkin_ispire}). Let $H=A\wr B$ be the wreath product of the groups $A$ and $B$. 
\begin{enumerate}
\item When $H$ is $\qq$-compact?
\item When $H$ is $\qq$-compact but not equationally Noetherian (a group is equationally Noetherian if the subsystem $\Ss^\prime\subseteq\Ss$ in~(\ref{eq:q_compact2}) does not depend on an equation $w(X)=1$)?
\item Is $H$ necessarily $\qq$-compact if both $A,B$ $\qq$-compact?
\end{enumerate}

\medskip

Let us explain the assertion of the problem above. Originally, B.~Plotkin posed it for group equations in the ``standard'' language $\LL=\{\cdot,^{-1},1\}$. However, in~\cite{shev_shah} the Problem was partially solved for languages with constants.

\begin{theorem}\textup{\cite{shev_shah}}
If a group $A$ is not abelian and $B$ is infinite, then $H$ is not $\qq$-compact in the language with constants $\LL(H)=\LL\cup\{h\mid h\in H\}$. 
\end{theorem} 

\medskip

Thus, for the language $\LL(H)$ the following problem remains open.

\medskip

\noindent{\bf Problem.} Let us consider the class of $\LL(H)$-equations. Is $H$ $\qq$-compact (equationally Noetherian) for abelian $A$? 

\medskip

In the conclusion of the whole paper, we should discuss other ways to solve Problems 4.4.7,~5.3.1-4 from~\cite{DMR_monograph}. Usually (see~\cite{DMR_monograph}), the negative solution of a problem in universal algebraic geometry may be found in structures of pure relational languages, since such languages admit a very simple view of equations. 

However, we cannot solve Problems~5.3.1-4 in relational languages. For this reason, we had to develop the algebraic geometry over equations with automorphisms. Thus, one can formulate a problem.

\medskip

\noindent{\bf Problem.} Is there an algebraic structure $\Acal$ of pure relational language $L$ such that 
\begin{enumerate}
\item $\Acal$ is an $L$-equational domain,
\item $\Acal$ is $\qq$-compact,
\item $\Acal$ is not $\uu$-compact?
\end{enumerate}

The information of the author:

Artem N. Shevlyakov

Sobolev Institute of Mathematics

644099 Russia, Omsk, Pevtsova st. 13

Phone: +7-3812-23-25-51.

\bigskip

Omsk State Technical University

644050 Russia, Omsk, pr. Mira, 11

\bigskip

e-mail: \texttt{a\_shevl@mail.ru}
\end{document}